\documentclass{amsart}
\usepackage[ansinew]{inputenc}
\usepackage{amssymb}
\usepackage{amsmath}
\usepackage{latexsym}
\usepackage{stmaryrd}
\usepackage{hyperref}

\newtheorem{thm}{Theorem}[section]
\newtheorem{con}[thm]{Condition}
\newtheorem{cor}[thm]{Corollary}
\newtheorem{lem}[thm]{Lemma}
\newtheorem{prop}[thm]{Proposition}
\newtheorem{defn}[thm]{Definition}

\newtheorem{conj}[thm]{Conjecture}
\newtheorem*{thmA}{Theorem A}
\newtheorem*{thmB}{Theorem B}
\newtheorem*{scon}{Schanuel condition}

\newcommand{\R}{\mathbb R}

\newcommand{\qt}{\mathbb{Q}(\tau)}
\newcommand{\td}[2]{\textrm{td}_{\qt}({#2/#1})}
\newcommand{\tdo}[1]{\textrm{td}_{\qt}(#1)}
\newcommand{\cl}[1]{\textbf{cl}_{#1}}

\newcommand{\ld}[2]{\textrm{m.dim}_{\mathbb{Q}(\tau)}(#2/#1)}
\newcommand{\lf}[2]{\textrm{m.dim}_{\mathbb{Q}}(#2/#1)}
\newcommand{\lk}[2]{\textrm{m.dim}_{K}(#2/#1)}
\newcommand{\ldo}[1]{\textrm{m.dim}_{\mathbb{Q}(\tau)}(#1)}
\newcommand{\lfo}[1]{\textrm{m.dim}_{\mathbb{Q}}(#1)}
\newcommand{\Q}{\mathbb{Q}}
\newcommand{\Z}{\mathbb{Z}}
\newcommand{\N}{\mathbb{N}}

\newcommand{\oqt}{\qt^{ac}}

\def \<{\langle}
\def \>{\rangle}

\title[Real field with a power function and a subgroup]
 {The real field with an irrational power function and a dense multiplicative subgroup} 

\author{Philipp Hieronymi}

\thanks{This research was funded by \emph{Deutscher Akademischer
Austausch Dienst Doktorandenstipendium} and \emph{Engineering and
Physical Sciences Research Council}.}

\address{Mathematical Institute,
University of Oxford,
24-29 St Giles',
Oxford, OX1 3LB, United Kingdom }

\curraddr{\textsc{Department of Mathematics \& Statistics,
McMaster University, 1280 Main Street West
Hamilton, ON, L8S 4K1, Canada}}

    \email{P@hieronymi.de}

\begin{document}
 \maketitle
\begin{abstract}
This paper provides a first example of a model theoretically well behaved structure consisting of a proper o-minimal expansion of the real field and a dense multiplicative subgroup of finite rank. Under certain Schanuel conditions, a quantifier elimination result will be shown for the real field with an irrational
power function $x^{\tau}$ and a dense multiplicative subgroup of finite rank whose elements are algebraic over $\qt$.
Moreover, every open set definable in this structure is already definable in the reduct given by just the real field and the irrational power function.
\end{abstract}

\begin{section}{Introduction}
 Let $\tau \in \mathbb{R}\setminus \Q$. We will consider the multiplicative group $(\R_{>0},\cdot)$ as a $\qt$-linear space where  the multiplication is given by $a^q$ for every $q \in \qt$ and $a \in \R_{>0}$.

\begin{scon} \label{sc} Let $n \in \mathbb{N}$ and $a \in \mathbb{R}^n$, then
\begin{equation*}
\tdo{a} + \ldo{a} \geq \lfo{a},
\end{equation*}
where $\tdo{a}$ is the transcendence degree of $a$ over $\qt$, $\ldo{a}$ and $\lfo{a}$ are the dimensions of the $\mathbb{Q}(\tau)$- and $\mathbb{Q}$-linear subspaces of $\mathbb{R}_{>0}$ generated by $a$.
\end{scon}

\noindent Let $\overline{\mathbb{R}}=(\mathbb{R},<,+,\cdot,0,1)$ be the field
of real numbers and let $x
^{\tau}$ be the function
on $\mathbb{R}$ sending $t$ to $t^{\tau}$ for $t>0$ and to $0$ for $t \leq 0$. Let $\oqt$ be the algebraic closure of $\qt$. The main result of this paper is the following:

\begin{thmA}\label{introthm1} Let $\tau \in \mathbb{R}$ satisfy the Schanuel condition and let $\Gamma$ be a dense subgroup of $\mathbb{R}_{>0}$ of finite rank with $\Gamma \subseteq \oqt$. Then every definable set in $(\overline{\R},x^{\tau},\Gamma)$ is a boolean combination of sets of the form
\begin{equation*}
\bigcup_{g \in \Gamma^n} \{ x \in \mathbb{R}^m \ : \ (x,g) \in S\},
\end{equation*}
where $S \subseteq \R^{m+n}$ is definable in $(\overline{\R},x^{\tau})$. Moreover, every open set definable in $(\overline{\R},x^{\tau},\Gamma)$ is already definable in $(\overline{\R},x^{\tau})$.
\end{thmA}
\noindent A finite rank subgroup of $\R_{>0}$ is a subgroup that is contained in the divisible closure of a finitely generated subgroup.
In fact, we will prove Theorem A not only for finite rank subgroups, but also for subgroups whose divisible closure has the Mann property (see page \pageref{mannproperty} for a definition of the Mann property). By work of Bays, Kirby and Wilkie in \cite{bkw} the Schanuel condition holds for co-countably many real numbers $\tau$. Assuming Schanuel's conjecture, the Schanuel condition also holds when $\tau$ is algebraic (see page \pageref{schanuelconj} for a statement of Schanuel's conjecture).\\

\noindent The significance of Theorem A comes from the fact that it produces the first example of a model theoretically well behaved structure consisting of a proper o-minimal expansion of the real field and a dense multiplicative subgroup of finite rank. So far it was only known by work of van den Dries and Günayd\i n in \cite{dense} that Theorem A holds if $(\overline{\R},x^{\tau})$ is replaced by $\overline{\R}$. In particular, by \cite{ayhanme2}, every open set definable in an expansion of the real field by a dense multiplicative subgroup $\Gamma$ of $\R_{>0}$ of finite rank is semialgebraic. However Tychonievich showed in \cite{michael} that the structure $(\overline{\mathbb{R}},\Gamma)$ expanded by the restriction of the exponential function to the unit interval defines the set of integers and hence every projective subset of the real line. Such a structure is as wild from a model theoretic view point as it can be. In contrast to this, every expansion of the real field whose open definable sets are definable in an o-minimal expansion, can be considered to be well behaved. For details, see Miller \cite{tame} and Dolich, Miller and Steinhorn \cite{dms}.\\

\noindent None of the assumptions of Theorem A can be dropped. By \cite{me}, Corollary 1.5, $(\overline{\R},x^{\tau},2^{\mathbb{Z}})$ defines the set of integers.
For $\tau=\log_2(3)$, the Schanuel condition fails. Since $(\overline{\R},x^{\log_2(3)},2^{\mathbb{Z}}3^{\mathbb{Z}})$ defines $2^{\mathbb{Z}}$, it also defines $\Z$. On the other hand, for a non-algebraic real number $\tau$ satisfying the Schanuel condition such that $2^{\tau}$ is not in $\oqt$, we have again that  $2^{\mathbb{Z}}$ is definable $(\overline{\R},x^{\tau},2^{\mathbb{Z}}2^{\tau\mathbb{Z}})$ and so is $\mathbb{Z}$.\\

\noindent However, Theorem A holds for certain multiplicative subgroups containing elements that are not algebraic over $\qt$.
\begin{thmB}\label{introthm2} Let $\tau \in \mathbb{R}$ satisfy the Schanuel condition, let $a_1,...,a_n\in \oqt$ and let $\Delta$ be the $\qt$-linear subspace of $(\R_{>0},\cdot)$ generated by $a_1,...,a_n$.
Then every definable set in $(\overline{\R},x^{\tau},\Delta)$ is a boolean combination of sets of the form
\begin{equation*}
\bigcup_{g \in \Delta^n} \{ x \in \mathbb{R}^m \ : \ (x,g) \in S\},
\end{equation*}
where $S \subseteq \R^{m+n}$ is definable in $(\overline{\R},x^{\tau})$. Moreover, every open set definable in $(\overline{\R},x^{\tau},\Delta)$ is already definable in $(\overline{\R},x^{\tau})$.
\end{thmB}

\subsection{Acknowledgements} I am deeply indebted to my thesis supervisor Alex Wilkie. His
guidance, insights and encouragement made this work possible. I thank Martin Bays, Juan Diego Caycedo, Ayhan Günayd\i n, Chris Miller and Boris Zilber for helpful discussions on the topic of this paper.

\subsection{Coventions and notations} Above and in the rest of the paper $l,m,n$ always denote natural numbers. Also as usual `definable' means `definable with parameters' and when we want to make the language and the parameters explicit we write $\mathfrak{L}$-$B$-{\it definable} to mean definable in the appropriate $\mathfrak{L}$-structure using parameters from the set $B$. \\

\noindent In all instances, $K$ will be either $\Q$ or $\Q(\tau)$ and $\Gamma$ will always denote a multiplicative subgroup of $\R_{>0}$. Further, every linear space considered in this paper will be a linear subspace of $(M_{>0},\cdot)$ and \emph{not} of $(M,+)$, where $M$ is a real closed field. In the case that $M_{>0}$ is a $K$-linear space, we will write $\lk{S_0}{S_1}$ for the $K$-linear dimension of the quotient linear space of the $K$-linear space generated by $S_0 \cup S_1$ and $S_0$, where $S_0,S_1\subseteq M_{>0}$. \\

\noindent For a given variety $W$, we will write $\dim W$ for its dimension. For sets $X_0,X_1$ in a field extending $\qt$ we will write $\td{X_0}{X_1}$ for the transcendence degree of the field extension $\qt(X_1\cup X_0)/\qt(X_0)$.\\

\noindent As usual, for any subset $S \subseteq X \times Y$ and $x \in X$, we write $S(x)$ for the set
\begin{equation*}
\{ y \in Y \ : \ (x,y)\in S \}.
\end{equation*}
For a subset $S \subseteq X^n$, $x \in S$ and a projection $\pi: X^n \to X^l$, we write $S(\pi(x))$ for the set
\begin{equation*}
\{ y \in S : \pi(y) = \pi(x) \}.
\end{equation*}

\subsection{O-minimality}
\noindent Let $\tau \in
\mathbb{R}\setminus \mathbb{Q}$ and let $x
^{\tau}$ be the function
on $\mathbb{R}$ sending $t$ to $t^{\tau}$ for $t>0$ and to $0$ for $t \leq 0$. In this paper we consider the structure $(\overline{\R},x^{\tau},\tau)$. We write $T$ for its theory and $\mathfrak{L}$ for its language. In \cite{miller} Miller showed that the
theory $T$ is o-minimal and model complete. \\

\noindent In the rest of this paper only the following facts about the o-minimality of $T$ will be used:\\

\noindent Let $M$ be a model of $T$. A definable subset $C$ of $M^n$ is a \emph{cell}, if for some projection $\pi: M^n \to M^m$, $\pi$ is a homeomorphism of $C$ onto its image and $\pi(C)$ is open. Since $T$ is o-minimal, every definable set $X\subseteq M^n$ is a finite union of cells which are defined over the same parameter set.\\

\noindent Let $A$ be any subset of $M$. We write $\cl{T}(A)$ for the definable closure of $A$ in $M$. By o-minimality of $T$, $\cl{T}(A)$ is itself a model of $T$. Moreover, the function $\cl{T}(-)$ is a \emph{pregeometry}; that is for every $A \subseteq M$, $a \in A$ and $b\in M$

\begin{itemize}
  \item [(i)] $A \subseteq \cl{T}(A)$,
  \item [(ii)] $b \in \cl{T}(A)$ iff $b \in \cl{T}(A_0)$, for some finite $A_0\subseteq A$,
  \item [(iii)] $\cl{T}(\cl{T}(A))=\cl{T}(A)$,
  \item [(iv)] if $b \in \cl{T}(A)\setminus \cl{T}(A\setminus \{a\})$, then $a \in \cl{T}\big((A\setminus \{a\})\cup\{b\}\big)$.
\end{itemize}
Property (iv) is called \emph{the Steinitz exchange principle}.\\
\noindent For two subsets $A,B \subseteq M$, we will say that $A$ is \emph{$\cl{T}$-independent} over $B$ if for every $a\in A$ $$a \notin \cl{T}\big(B \cup (A\setminus \{a\})\big).$$\\
\noindent Let $M,M'$ be two models of $T$, let $N\preceq M, N' \preceq M'$ and let $\beta: N \to N'$ be an $\mathfrak{L}$-isomorphism. Let $a \in M$ and $b\in M'$ be such that $a < c$ iff $b < \beta(c)$, for all $c\in N$. Then there is an $\mathfrak{L}$-isomorphism $\beta': \cl{T}(N,a) \to \cl{T}(N',b)$ extending $\beta$ and sending $a$ to $b$.

\end{section}

\begin{section}{A Schanuel condition and the Mann property}
 Let $\tau \in \mathbb{R}$. As above, we will consider $(\R_{>0},\cdot)$ as a $\qt$-linear space. For $a \in \R_{>0}^n$, we write $\ldo{a}$ and $\lfo{a}$ for the dimensions of the $\mathbb{Q}(\tau)$- and $\mathbb{Q}$-linear subspaces of $\mathbb{R}_{>0}$ generated by $a$.

\begin{con} \label{schanuel} Let $n \in \mathbb{N}$ and $a \in \mathbb{R}^n$, then
\begin{equation*}
\tdo{a} + \ldo{a} \geq \lfo{a}.
\end{equation*}
\end{con}

\noindent This condition has been analysed in \cite{bkw}. The main theorem of \cite{bkw} states that Condition \ref{schanuel} holds for co-countably many real numbers.

\begin{thm}{(\cite{bkw} Theorem 1.3)} Let $\tau \in \mathbb{R}$. If $\tau$ is not $\emptyset$-definable in $(\overline{\mathbb{R}},\exp)$, then Condition \ref{schanuel} holds.
\end{thm}

\noindent It is not known whether there is any other irrational number $\tau$ such that Condition \ref{schanuel} holds. However it follows easily from a famous open conjecture of Schanuel, that every \emph{algebraic} real number $\tau$ satisfies Condition \ref{schanuel}. The conjecture states as follows.

\begin{conj}{(Schanuel's Conjecture)}\label{schanuelconj} Let $n \in \mathbb{N}$ and $a \in \mathbb{R}^n$, then
\begin{equation*}
\textrm{td}_{\Q}(a,\exp(a)) \geq \lfo{\exp(a)}.
\end{equation*}
\end{conj}

\subsection{The Mann property}
\label{mannproperty}
In this section we consider the Mann property and its connection to Condition \ref{schanuel}. Let $F$ be a field, $E$ be a subfield of $F$ and $G$ be any
subgroup of the multiplicative group $F^{\times}$. Consider
equations of the form
\begin{equation}\label{mannequation}
a_1 x_1 + ... + a_n x_n = 1,
\end{equation}
where $a_1,...,a_n \in E$. We say a solution $(g_1,...,g_n)\in G^n$ of \eqref{mannequation}
 is \emph{non-degenerate} if for every
non-empty subset $I$ of $\{1,...,n\}$, $\sum_{i \in I} a_ig_i \neq
0$. Further we say that $G$ has the \emph{Mann property over $E$} if every
equation of the above type \eqref{mannequation} has only finitely many non-degenerate
solutions in $G^n$. We also call an element $g \in G^n$ a \emph{Mann
solution of $G$ over $E$} if it is a non-degenerate solution in $G^n$ of an equation of
the form \eqref{mannequation}.\\
\noindent  In fact, it follows from work of Evertse in
\cite{Evertse} and van der Poorten and Schlickewei in \cite{Mann}
that every finite rank multiplicative subgroup of a field of characteristic 0 has the Mann property over $\Q$.
Combining this with \cite{dense}, Proposition 5.6, we get the following theorem.

\begin{thm}\label{mannproperty} Every finite rank multiplicative subgroup of $\R_{>0}$ has the Mann property over $\qt$.
\end{thm}

\noindent We conclude this section by showing that under Condition \ref{schanuel} the $\qt$-linear space generated by a \emph{divisible} multiplicative subgroup $\Gamma$ has the Mann property over $\qt$, if $\Gamma$ has the Mann property over $\qt$ and every element of $\Gamma$ is algebraic over $\qt$.

\begin{prop}\label{mannfor2qt} Assume Condition \ref{schanuel} holds for $\tau$. Let $\Gamma$ be a divisible multiplicative subgroup of $\mathbb{R}_{>0}$ with $\Gamma \subseteq \oqt$ and $\Delta$ be the $\qt$-linear subspace of $\R_{>0}$ generated by $\Gamma$. Then
\begin{itemize}
\item[(i)] every  Mann solution of $\Delta$ over $\qt$ is in $\Gamma$ and
\item[(ii)] $\Delta$ has the Mann property over $\qt$, if $\Gamma$ has the Mann property over $\qt$.
\end{itemize}
\end{prop}
 \begin{proof} It is enough to show (i). Therefor let $a_1,...,a_n \in \qt$ and let $g=(g_1,...,g_n) \in \Delta^n$ be such that
 \begin{equation}\label{manneq}
a_1 g_1 + ... + a_n g_n = 1
 \end{equation}
and for all $I \subseteq \{1,...,n\}$ \begin{equation}\label{eqirr}
\sum_{i\in I} a_i g_i \neq 0.
\end{equation}
We will show that $g \in \Gamma^n$.\\
\noindent Let $h \in \Gamma^m$ be such that $\ld{h}{g}=0$ and $\lfo{h}=m$. Let $k$ be the maximal natural number such that there is a subtuple $g'$ of $g$ of length $k$ such that $\lf{h}{g'}=k$. It just remains to verify that $k=0$. For a contradiction, suppose that $k>0$. By \eqref{manneq} and \eqref{eqirr}, we have that
\begin{equation*}
\td{h}{g'} < k.
\end{equation*}
Since every coordinate of $h$ is algebraic over $\qt$,
\begin{equation*}
\tdo{h,g'} + \ldo{h,g'} < k + m = \lfo{h,g'}.
\end{equation*}
This contradicts Condition \ref{schanuel}.
\end{proof}

\end{section}
\begin{section}{Tori and special pairs}

Let $M$ be a model of $T$. In the following we will consider $(M_{>0},\cdot)$ as a $K$-linear space where $K$ is either $\mathbb{Q}$ or $\qt$ and the multiplication is given by $a^q$ for every $q \in K$ and $a \in M_{>0}$.

\begin{defn} A basic \emph{$K$-torus} $L_0$ of $M^n$ is the set of solutions of equations of the form
\begin{align*}
 x_1^{p_{1,1}} \cdot ... \cdot  x_n^{p_{1,n}} &= 1,\\
\vdots   \ \ \  \  \ \ \ \ \ \ & \vdots\\
x_1^{p_{l,1}} \cdot ... \cdot  x_n^{p_{l,n}} &= 1,
\end{align*}
where $p_{i,j} \in K$.\\
For $b\in M^m$, a \emph{$K$-torus} $L$ over $b$ is a subset of $M_{>0}^n$ of the form $L_0(b)$, for some basis $K$-torus $L_0$ of $M_{>0}^{m+n}$. We will write $\dim L$ for the dimension of $L$ which is given by the corank of the matrix $(p_{i,j})_{i=1,...,l,j=m+1,...,m+n}$.
\end{defn}

\noindent The dimension of a torus and the linear dimension of a tuple in $M_{>0}$ correspond to each other. Let $a \in M^n$, $b \in M^m$ and let $L$ be the minimal $\qt$-torus over $b$ containing $a$. Then
$$
\dim L = \ld{b}{a}.
$$
For the following, the reader is reminded that for a set $S \subseteq M^n$, $y \in S$ and a projection $\pi: M^n \to M^l$
we write $S(\pi(y))$ for the set
$$
                                \{ z \in S \ : \ \pi(z)=\pi(y)\}.
$$

\begin{defn}\label{defspecial} Let $W\subseteq M^n$ be a variety and let $L$ be a $\qt$-torus. The pair $(W,L)$ is called \emph{special}, if $n=0$ or
\begin{equation*}
\dim W(\pi(y)) + \dim L(\pi(y)) < n-l,
\end{equation*}
for every point $y \in W\cap L$ and every projection $\pi : M^n \to M^l$, where $l\in \{0,...,n\}$.
\end{defn}
\noindent Note that the notion of specialness is first order expressible. In particular, for given variety $W\subseteq M^{m+n}$ and $\qt$-torus $L\subseteq M^{m+n}$, the set
$$
\{ a \in M^m \ : \ (W(a),L(a)) \textrm{ is special } \}
$$
is $\mathfrak{L}$-definable, where $\mathfrak{L}$ is the language of $(\overline{\R},x^{\tau},\tau)$.

\subsection{A Mordell-Lang Theorem for special pairs}
Let $\Gamma$ be a multiplicative subgroup of $\mathbb{R}_{>0}$ such that the divisible closure of $\Gamma$ has the Mann property over $\qt$ and $\Gamma$ is a subset of $\oqt$, the algebraic closure of $\qt$. Further,
\begin{equation*}
\textrm{ \emph{let $\Delta$ be the $\qt$-linear subspace of $\R_{>0}$ generated by $\Gamma$.}}
\end{equation*}
In this subsection we will prove the following theorem about special pairs defined over parameters from $\Delta$. Its statement is similar to a conjecture of Mordell and Lang.

\begin{thm}\label{specialml} Assume Condition \ref{schanuel}. Let $W\subsetneq \mathbb{R}^{l+n}$ be a variety defined over $\qt$ and let $L\subseteq \mathbb{R}^{l+n}$ be a basic $\qt$-torus. Then there are finitely many basic $\mathbb{Q}$-tori $L_1,...,L_m$ and $g_1,...,g_m\in \Gamma^{l+n}$ such that
\begin{equation}\label{specialmleq}
\{(h,y) \in \Delta^{l} \times \mathbb{R}^n : \ (W(h),L(h)) \textrm{\emph{ is special }}\}\cap W  \subseteq \bigcup_{i=1}^{m} g_i \cdot L_i.
\end{equation}
and $\dim L_i(z)<n$, if $n>0$, for every $z \in \R^{l}$ and $i=1,...,m$.
\end{thm}

\noindent For the proof of Theorem \ref{specialml} the following lemma is needed.

\begin{lem}\label{specialindelta} Assume Condition \ref{schanuel}. Let $g \in \Delta^l$, $y \in \mathbb{R}^n$ and let $W$ be a variety defined over $\qt(g)$ and $L$ be an $\qt$-torus over $g$. If the pair $(W,L)$ is special and $y\in W\cap L$, then $y \in \Delta^n$.
\end{lem}
\begin{proof} Let $y=(y_1,...,y_n) \in W\cap L$. If $(W,L)$ is special, we have that for every subset $I\subseteq \{1,...,n\}$
\begin{equation}\label{dimineq}
\td{g,(y_i)_{i\in I}}{(y_j)_{j \notin I}} + \ld{g,(y_i)_{i\in I}}{(y_j)_{j \notin I}} < n -  |I|.
\end{equation}
For a contradiction suppose that $y \notin \Delta^n$. We easily can reduce to the case that $g,y$ are multiplicatively independent, ie. $\lfo{g,y}=l+n$. By \eqref{dimineq}
\begin{equation*}
\td{g}{y} + \ld{g}{y} < n.
\end{equation*}
By definition of $\Delta$, we can assume there is $s \in \N$ and a subtuple $h \in (\Delta \cap \oqt)^s$ of $g$ such that
$$
\ld{h}{g} = 0.
$$
Hence
\begin{align*}
\tdo{g,y} &+ \ldo{g,y}\\
&= \tdo{g} + \ldo{g} + \td{g}{y} + \ld{g}{y}\\
&< l-s + s + n = l+n.
\end{align*}
By Condition \ref{schanuel}, $\lfo{g,y} < l +n$. This is a contradiction to our assumption on $g$ and $y$. \end{proof}

\noindent In \cite{dense} it is shown that the Mann property implies the Mordell-Lang property. In our notation \cite{dense}, Proposition 5.8, is stated as follows:

\begin{lem}\label{mordelllangfact} Let $G$ be a multiplicative subgroup of $\R_{>0}$ with the Mann property over $\qt$. Then for every variety $W\subseteq \mathbb{R}^n$, there  are finitely many basic $\Q$-tori $L_1,...,L_m$ of $\mathbb{R}_{>0}^n$ and $g_1,...,g_m\in G^n$ such that
\begin{equation*}\label{mlfacteq}
W \cap G^n = \bigcup_{i=1}^m g_i\cdot L_i \cap G^n.
\end{equation*}
Moreover, every coordinate of $g_1,...,g_n$ is a coordinate of a Mann solution of $G$ over $\qt$.
\end{lem}

\noindent The fact that every coordinate of $g_1,...,g_n$ is a coordinate of a Mann solution over $\qt$ is not in the statement of \cite{dense}, Proposition 5.8, but explicit in its proof.

\begin{proof}[Proof of Theorem \ref{specialml}] Since the divisible closure of $\Gamma$ has the Mann property over $\qt$, $\Delta$ has the Mann property over $\qt$ by Proposition \ref{mannfor2qt}(ii). Hence by Lemma \ref{mordelllangfact} there are basic $\Q$-tori $L_1,...,L_m$ and $g_1,...,g_m\in \Delta^{l+n}$ such that
\begin{equation}\label{specialmleq2}
W \cap \Delta^{l+n} = \bigcup_{i=1}^m g_i \cdot L_i \cap \Delta^{l+n}.
\end{equation}
By Proposition \ref{mannfor2qt}(i), every Mann solution over $\qt$ of $\Delta$ is in the divisible closure of $\Gamma$. Hence every coordinate of $g_1,...,g_m$ is in the divisible closure of $\Gamma$ by Lemma \ref{mordelllangfact}. After changing the $L_i$'s slightly, we can even take $g_1,...,g_m\in \Gamma^{l+n}$.
Finally, the left hand side of \eqref{specialmleq} in the statement of the theorem is contained in $\Delta^{l+n}$ by Lemma \ref{specialindelta}.\\
\noindent For the second statement of the theorem, let $i\in\{1,...,m\}$ and let $(h,y)$ be in the intersection of the left hand side of \eqref{specialmleq} and $g_i \cdot L_i$. Since $(W(h),L(h))$ is special, we have $\dim W(h)<n$. Hence $\dim L_i(h) < n$ by \eqref{specialmleq2}. It follows directly that $\dim L_i(z)<n$ for every $z\in \R^l$.
\end{proof}

\end{section}

\begin{section}{The axiomatization}\label{definitions}
Let $\Gamma$ be a multiplicative subgroup of $\mathbb{R}_{>0}$ such that the divisible closure of $\Gamma$ has the Mann property over $\qt$ and $\Gamma$ is a subset of $\oqt$. Let $\Delta$ be the $\qt$-linear subspace of $\R_{>0}$ generated by $\Gamma$. Further we assume that
\begin{equation}\label{finitelymanycosets}
|\Gamma:\Gamma^{[d]}| < \infty, \textrm{ for every $d \in \mathbb{N}$},
\end{equation}
where $\Gamma^{[d]}$ is the group of $d$th powers of $\Gamma$. In the rest of this section, axiomatizations of $(\overline{\R},x^{\tau},\tau,\Gamma)$ and $(\overline{\R},x^{\tau},\tau,\Delta)$ will be given.\\

\noindent Note that \eqref{finitelymanycosets} holds for every multiplicative subgroup of $\mathbb{R}_{>0}$ which has finite rank.

\subsection{Abelian subgroups} Let $G$ be a multiplicative subgroup of $(M_{>0},\cdot)$ for some real closed field $M$. For $k=(k_1,...,k_n)\in\Z^n$ and $g=(g_1,\dots,g_n) \in G^n$, we define
\begin{equation*}
\chi_{k}(g):= g_1^{k_1}\cdot...\cdot g_n^{k_n}.
\end{equation*}
Also, for $m \in \Z$, we will write
$$
D_{k,m} := \{  g \in G^n : \chi_{k}(g) \in G^{[m]}\}.
$$
Note that $(G^{[m]})^n\subseteq D_{k,m}$. Hence whenever $G^{[m]}$ is of finite index in $G$ we have that $D_{k,m}$ is of finite index in $G^n$. This implies that both $D_{k,m}$ and $G^n\setminus D_{k,m}$ are finite unions of cosets of $(G^{[m]})^n$.  Using the fact that the collection $\big\{(G^{[m]})^n:m\in \mathbb{N}\big\}$ is a distributive lattice of subgroups of $G^n$, we get the following consequence.

\begin{lem}\label{unionofcosets}
Let $n>0$, $k_1,\dots,k_s\in\Z^n$ and $m_1,\dots,m_t \in \mathbb{N}$.
Suppose that $|G:G^{[m_j]}|$ is finite for $j=1,\dots,t$. Then every boolean combination of cosets of $D_{k_i,m_j}$ in $G^n$ with $i\in\{1,\dots,s\}$ and $j\in\{1,\dots,t\}$ is a finite union of cosets of $(G^{[l]})^n$, where $l$ is the lowest common multiple of $m_1,\dots,m_t$.
\end{lem}

\noindent We say a subgroup $H$ of $G$ is \emph{pure}, if $h \in H^{[n]}$ whenever $h\in G^{[n]}$ for $n\in \N$.\\
For a pure subgroup $H$ of $G$ and a subset $A$ of $G$, we define $H_G\<A\>$ as the set of $g\in G$ such that $g^n$ is in the subgroup of $G$ generated by $H$ and $A$ for some $n > 0$; that is there are $h \in H$, $a\in A^t$, and $k \in
\Z^t$, such that $g^n = h \cdot \chi_k(a)$. Note that $H_G\<A\>$ is the smallest pure subgroup of $G$ containing $A$ and $H$.

\subsection{Languages and Mordell-Lang axioms for special pairs}
 Let $\mathfrak{L}$ be the language of $(\overline{\R},x^{\tau},\tau)$. We define the language $\mathfrak{L}_{\Gamma}$ as $\mathfrak{L}$ augmented by a constant symbol $\dot{\gamma}$ for every $\gamma \in \Gamma$. The $\mathfrak{L}$-structure $(\overline{\R},x^{\tau},\tau)$ naturally becomes a $\mathfrak{L}_{\Gamma}$-structure by interpreting every $\gamma\in \Gamma$ as $\dot{\gamma}$. Let $T_{\Gamma}$ be the theory of this $\mathfrak{L}_{\Gamma}$-structure. Finally let $\mathfrak{L}_{\Gamma}(U)$ be the language $\mathfrak{L}_{\Gamma}$ expanded by an unary predicate symbol $U$.\\

\noindent Let $W$ be a variety defined over $\qt$ and let $L$ be a basic $\qt$-torus. Note that both $W$ and $L$ are $\mathfrak{L}$-$\emptyset$-definable. Further let $\varphi$ be the $\mathfrak{L}_{\Gamma}(U)$-formula which defines the set
$$S:= \{(g,y) \in \Gamma^l \times \mathbb{R}^n : \ (g,y) \in W \textrm{ and } (W(g),L(g)) \textrm{ is special} \}.$$
By Theorem \ref{specialml}, there are basic $\Q$-tori $L_1,...,L_m$ and $\gamma_1,...,\gamma_m \in \Gamma^{l+n}$ such that $S$ is a subset of the union of $\gamma_1\cdot L_1,...,\gamma_m \cdot L_m$ and $\dim L_i(z)<n$ for every $i=1,...,m$ and $z \in \R^l$. Let $k_{i,1},...,k_{i,s_i} \in \mathbb{Z}^{l+n}$ be such that
$$
L_i = \{ x \in \mathbb{R}^n \ : \ \chi_{k_{i,j}}(x)=1, \textrm{ for } j=1,...,s_i  \}.
$$
The \emph{Mordell-Lang axiom of $(W,L)$} is defined as the $\mathfrak{L}_{\Gamma}(U)$-formula $\psi_{(W,L)}$ given by
$$
\varphi(x) \rightarrow \bigvee_{i=1}^{m} \bigwedge_{j=1}^{s_i} \chi_{k_{i,j}}(x)=\chi_{k_{i,j}}(\gamma_i).
$$

\begin{subsection}{The theory}
We consider the class of all $\mathfrak{L}_{\Gamma}(U)$-structure $(M,G)$ satisfying the
 following axioms:
\smallskip
\begin{itemize}
  \item [(A1)] $M$ is a model of $T_{\Gamma}$,
  \item [(A2)] $G$ is a dense multiplicative subgroup of $M$ with pure subgroup $\Gamma$,
  \item [(A3)] $|\Gamma
: \Gamma^{[n]}| = | G : G^{[n]}|$, for all $n \in \mathbb{Q}$,
\item[(A4)] $L\cap (G\setminus\{1\})^n = \emptyset$, for every basic $\qt$-torus $L\subseteq M^n$ which is not a basic $\Q$-torus.
\item[(A5)] Mordell-Lang axiom $\psi_{W,L}$ for every variety $W\subsetneq M^{l+n}$ over $\qt$ and every basic $\qt$-torus $L\subseteq M^{l+n}$,
\item[(A6)] the set \begin{equation*} \bigcap_{i=1}^{m}\{ a \in M \ : \ \forall
g \in G^l \ f_i(g,b)\neq a\} \end{equation*} is dense
in $M$, for all $b\in M^n$ and $\mathfrak{L}$-$\emptyset$-definable functions
$f_1,...,f_m: M^{l+n} \to M$.
\end{itemize}
\smallskip
\noindent One can easily show that there is a first order $\mathfrak{L}_{\Gamma}(U)$-theory whose models are exactly the structures satisfying (A1)-(A6). Let $T_{\Gamma}(\Gamma)$ be this $\mathfrak{L}_{\Gamma}(U)$-theory.

\begin{prop} Assume Condition \ref{schanuel}. Then $(\overline{\R},x^{\tau},\tau,\Gamma)\models T_{\Gamma}(\Gamma)$.
\end{prop}
\begin{proof}
The axioms (A1)-(A3) hold by definition. Axiom (A5) is implied by Theorem \ref{specialml}. Since $\Gamma$ is a subset of $\oqt$, it is countable and hence (A6) holds for $\Gamma$. Finally consider axiom (A4). Let $L$ be a basic $\qt$-torus $L\subsetneq \R^{n}$ which is not a $\Q$-torus.
For a contradiction, suppose there is $g \in (\Gamma\setminus\{1\})^n$ such that $g\in L$. Since every element of $\Gamma$ is algebraic over $\qt$ and $L$ is not a $\Q$-torus, we get
$$ \tdo{g} + \ldo{g} = 0 + \ldo{g} < \lfo{g}.$$
This contradicts Condition \ref{schanuel}.
\end{proof}

\noindent For an axiomatization of  $(\overline{\R},x^{\tau},\tau,\Delta)$, consider the $\mathfrak{L}_{\Gamma}(U)$-structures $(M,G)$ satisfying
\smallskip
\begin{itemize}
\item [(A7)] $G$ is a dense multiplicative subgroup of $M$ with subgroup $\Delta$,
\item[(A8)] $g^p \in G$, for every $g \in G$ and $p \in \qt$.
\end{itemize}
\smallskip
Let $T_{\Gamma}(\Delta)$ be the first order $\mathfrak{L}_{\Gamma}(U)$-theory whose models are exactly the structures satisfying (A1) and (A5)-(A8).

\begin{prop} Assume Condition \ref{schanuel}. Then $(\overline{\R},x^{\tau},\tau,\Delta)\models T_{\Gamma}(\Delta)$.
\end{prop}

\noindent Among other things, it will be shown in the next section that both $T_{\Gamma}(\Gamma)$ and $T_{\Gamma}(\Delta)$ are complete.

\end{subsection}

\end{section}

\begin{section}{Quantifier elimination}
In this section, the first part of Theorem A and Theorem B is proved. We continue with the notation fixed at beginning of the last section (see page \pageref{definitions}). In the following, $\tilde{T}$ is either $T_{\Gamma}(\Gamma)$ or $T_{\Gamma}(\Delta)$.\\

\noindent Let $x=(x_1,\dots,x_m)$ be a tuple of distinct variables. For every $\mathfrak{L}_{\Gamma}(U)$-formula $\varphi(x)$ of the form
\begin{equation}\label{formform}
\exists y_1 \cdots \exists y_{n}
\bigwedge_{j=1}^{n} U(y_i) \wedge \psi(x,y_1,\dots,y_{n}),
\end{equation}
where $\psi(x,y_1,\dots,y_{n})$ is an $\mathfrak{L}_{\Gamma}$-formula, let $U_{\varphi}$ be a new relation symbol of arity $m$. Let $\mathfrak{L}_{\Gamma}(U)^+$ be the
language $\mathfrak{L}_{\Gamma}(U)$ with relation symbols $U_{\varphi}$ for every $\varphi$ of the form \eqref{formform}.
Let $\tilde{T}^+$ be the $\mathfrak{L}_{\Gamma}(U)^+$-theory
extending the theory $\tilde{T}$ by axioms
\begin{equation*}
\forall x\big(U_{\varphi}(x) \leftrightarrow \varphi(x)\big),
\end{equation*}
for each $\varphi$ of the form (\ref{formform}). In order to show the first part of Theorem A and Theorem B, one has to show the following:
\begin{thm}\label{nearMC2}
The theory $\tilde{T}^+$ has quantifier elimination.
\end{thm}
\noindent The rest of this section will provide a proof of Theorem \ref{nearMC2}. In fact, we will give the proof only for $\tilde{T}=T_{\Gamma}(\Gamma)$. The case of $T_{\Gamma}(\Delta)$ can be handled in almost exactly the same way. We will comment on the differences at the end of this section.

\begin{subsection}{Main Lemma} This subsection establishes the main technical lemma used in the proof of Theorem \ref{nearMC2}. Therefor the following instance of Jones and Wilkie \cite{Thm51}, Theorem 4.2, is needed.

\begin{prop}\label{inst51}Let $M \models T$ and $b\in M, A \subseteq M$. If  $b\in \cl{T}(A)$,
then there are $y \in M^n$, a variety $W$ defined over $\qt(A)$ and an $\qt$-torus $L$ over $A$ such that $(b,y) \in W \cap L$ and
\begin{equation*}
\dim W + \dim L \leq n+1.
\end{equation*}
Further $y$ can be assumed to be multiplicatively independent over $b,A$, ie. for every $a \in A^m$ $$\lf{b,a}{y}=n.$$
\end{prop}

\begin{lem}\label{mainlemma} Let $(M,G)\models \tilde{T}$ and $H$ be a pure subgroup of $G$ containing all interpretations of the constants $\dot{\gamma}$, where $\gamma \in \Gamma$.
Then
\begin{equation*}
\cl{T}(H)\cap G = H.
\end{equation*}
\end{lem}
\begin{proof} The inclusion $H \subset \cl{T}(H) \cap G$ is trivial.
It is just left to show that whenever $g \in \cl{T}(H) \cap G$, then
$g$ is also in $H$. So let $g \in \cl{T}(H) \cap G$. By Proposition \ref{inst51}, there is $n \in \mathbb{N}$ such that there are $h \in (H\setminus\{1\})^m$, $y \in M^n$, a variety $W\subseteq M^{m+1+n}$ defined over $\qt$ and a basic $\qt$-torus $L\subseteq M^{m+1+n}$ such that
\begin{equation*}
(g,y) \in W(h)\cap L(h),
\end{equation*}
\begin{equation}\label{mainlemeq1}
\dim W(h) + \dim L(h) \leq  n+1
\end{equation}
and
\begin{equation}\label{mainlemeq2}
\lf{h',g}{y} = n, \textrm{ for every $h' \in H^l$}.
\end{equation}
Take $n$ minimal with this property.\\

\noindent We will now show that $n=0$. For a contradiction, suppose that $n>0$. We first prove that the pair $(W(h,g),L(h,g))$ is special. Towards a contradiction, suppose there are $z \in W(h,g)\cap L(h,g)$, $l<n$ and a projection $\pi: M^n \to M^l$ with
\begin{equation*}
\dim W(h,g)(\pi(z)) + \dim L(h,g)(\pi(z)) \geq n-l.
\end{equation*}
Let $W'\subseteq M^{l+1}$ be the variety defined by all polynomial equations over $\qt(h)$ which are satisfied by $(g,\pi(z))$ and let $L'\subseteq M^{l+1}$ be the smallest $\qt$-torus over $h$ which contains $(g,\pi(z))$. Then
\begin{align*}
\dim W' &+ \dim L'\\
&\leq \dim W(h) + \dim L(h)  - \dim W(h,g)(\pi(z)) -\dim L(h,g)(\pi(z))\\
&\leq n +1 -(n-l)=l+1.
\end{align*}
But this contradicts the minimality of $n$. Hence $(W(h,g),L(h,g))$ is special. By (A5), there are $\gamma \in \Gamma$ and a basic $\Q$-torus $L_0$ such that $(h,g,y)\in \gamma \cdot L_0$ and $\dim L_0(h,g)<n$. Hence $\lf{\gamma,h,g}{y}<n$. This is a contradiction against \eqref{mainlemeq2}. Hence $n=0$. \\

\noindent Since $n=0$, there is a variety $W\subseteq M$ defined over $\qt$ and a basic $\qt$-torus $L\subseteq M$ such that $g \in W(h)\cap L(h)$ and
\begin{equation}\label{mainlemeq3}
\dim W(h) + \dim L(h) \leq 1.
\end{equation}
First consider the case that $\dim W(h) =1$. By \eqref{mainlemeq3}, $\dim L(h) = 0$. By (A4) and $(h,g)\in L$, $L$ is a basic $\Q$-torus. Hence $\lf{h}{g}=0.$ Since $H$ is pure and $g\in G$, we have $g \in H$.\\
Now consider $\dim W(h)=0$. By Definition \ref{defspecial} of specialness and $(h,g)\in G^{m+1}$, the pair $(W(h,g),L(h,g))$ is special. By (A5), there are a basic $\Q$-torus $L_0$ and a $\gamma \in \Gamma$ such that $(h,g) \in \gamma \cdot L_0$. As above, we get $g \in H$.
\end{proof}

\begin{cor}\label{mainlemmacor} Let $(M,G)\models \tilde{T}$ and $H$ be a pure subgroup of $G$ containing all interpretations of the constants $\dot{\gamma}$, where $\gamma \in \Gamma$.
If $A$ is $\cl{T}$-independent over $G$, then
\begin{equation*}
\cl{T}(A,H)\cap G = H.
\end{equation*}
\end{cor}
\begin{proof} $H$ is
obviously a subset of $\cl{T}(A,H)\cap G$. By Lemma
\ref{mainlemma} it is only left to show that
\begin{equation}\label{eqmaincor}
\cl{T}(A,H)\cap G \subseteq \cl{T}(H) \cap G.
\end{equation}
So let $g \in \cl{T}(A,H) \cap G$ and $A'$ be a minimal subset of $A$
such that $g \in \cl{T}(A',H)\cap G$. For a contradiction, suppose
that $A'$ is non-empty and let $a \in A'$. By minimality of $A'$, we
have $g \notin \cl{T}(A'\setminus\{a\},H)$. But then the Steinitz Exchange
Principle implies that $a \in \cl{T}(A'\setminus\{a\},g,H)$. Since $g \in H
\subseteq G$, we get that
\begin{equation*}
a \in \cl{T}(A'\setminus\{a\},G).
\end{equation*}
But this is a contradiction to the $\cl{T}$-independence of $A$ over
$G$. Hence $A'$ is empty and $g \in \cl{T}(H) \cap G$. Thus
\eqref{eqmaincor} holds. \end{proof}

\end{subsection}

\begin{subsection}{Back and forth}

Let $(M,G),(M',G')$ be two
$(|\Gamma|)^+$-saturated models of $\tilde{T}$. Then $M,M'$ are
models of $T_{\Gamma}$. Let $\mathcal{E}$ be the set of all
$\mathfrak{L}_{\Gamma}$-elementary maps from $M$ to $M'$. Let
$\mathcal{S}$ be the set of all $\beta \in \mathcal{E}$ such that
there exist
\begin{itemize}
  \item a finite subset $A$ of $M$, and a finite subset $A'$ of
  $M'$,
  \item a pure subgroup $H$ of $G$ of cardinality at most
  $|\Gamma|$ and a pure subgroup $H'$ of $G'$ of cardinality at most
  $|\Gamma|$
\end{itemize}
such that
\begin{enumerate}
  \item $\beta$ is an $\mathfrak{L}_{\Gamma}(U)$-isomorphism between $(\cl{T}(A,H),H)$ and
  $(\cl{T}(A',H'),H')$,
  \item $A$ is $\cl{T}$-independent over $G$, and $A'$
  is $\cl{T}$-independent over $G'$ with $\beta(A)=A'$,
  \item $\Gamma$ is a pure subgroup of $H$ and $H'$.
\end{enumerate}
\medskip \noindent
By Corollary \ref{mainlemmacor}, $(\cl{T}(A,H),H)$ and $(\cl{T}(A',H'),H')$ are $\mathfrak{L}_{\Gamma}(U)$-substructures of $(M,G)$ and $(M',G')$ respectively. Hence every element of $\mathcal{S}$ is a partial isomorphism between $(M,G)$ and $(M',G')$.

\begin{lem}\label{bnf} The set $\mathcal{S}$ is a back-and-forth system.
\end{lem}
\begin{proof} In order to prove this statement, we will show that for
every $\beta \in \mathcal{S}$ and every $a\in M$, there is
a $\tilde{\beta} \in \mathcal{S}$ such that $\tilde{\beta}$ extends $\beta$ and
$a\in dom(\gamma)$. In fact, this is enough because of
the symmetry of the setting.\\
Let $\beta \in \mathcal{S}$ and $a \in M$. We can
assume that $a\notin dom(\beta)$. Further let $A,A',H,H'$ witness that $\beta \in \mathcal{S}$.\\
\smallskip\noindent
\underline{Case 1:} $a\in G$.\\
Let $p(x)$ be the $\mathfrak{L}_{\Gamma}(U)$-type consisting of the
$\mathfrak{L}_{\Gamma}$-type of $a$ over $\cl{T}(A,H)$ and for every $h \in H$, $k\in\Z$ and $n>0$ one of the formulas
\begin{align}\label{pureness}
x^k \cdot h &\in G^{[n]},\\
x^k \cdot h &\notin G^{[n]},
\label{pureness2}
\end{align}
depending on whether it is true in $(M,G)$ that $a^k h \in G^{[n]}$ or not.  Further let
$p'$ be the type over $\cl{T}(A',H')$ corresponding to $p$ via $\beta$. We want to find an
$a' \in M'$ such that $a'$ realizes $p'$. By compactness and saturation of
$(M',G')$, it is enough to show that finitely many formulas of
$p'$ can be satisfied. By o-minimality of $T$, this reduces to find an $a' \in M'$ with
\begin{equation}\label{conditionona}
(M',G') \models \beta(c) < a' < \beta(d) \wedge \bigwedge_{i=1}^n \phi_i(a'),
\end{equation}
for every $c,d \in \cl{T}(A,H)$ with $c<a<d$ and every finite collection of formulas $\phi_1,\dots,\phi_n$ of the form
\eqref{pureness} or \eqref{pureness2} with $(M,G)\models \bigwedge_{i=1}^n \phi_i(a)$.
By Lemma \ref{unionofcosets}, the set
$$Y:=\{ g \in G' \ : \ (M',G')\models \bigwedge_{i=1}^n \phi_i(g)\}$$
is a finite union of cosets of $G'^{[s]}$ in $G'$ for some $s \in \mathbb{N}$. Since $G'^{[s]}$ is dense in $G'$, we have that $Y$ is dense in $G'$ as well. Since $G'$ is dense in $M'$, we have that $Y\cap \big(\beta(c),\beta(d)\big)$ is dense in $\big(\beta(c),\beta(d)\big)$. Now take any $a' \in Y \cap \big(\beta(c),\beta(d)\big)$. This $a'$ satisfies \eqref{conditionona}.\\
By definition, $H_G\langle a \rangle$ and $H'_{G'}\langle a' \rangle$ are the smallest pure subgroups of $G$ and $G'$ containing $H\cup \{a\}$ and $H' \cup \{a'\}$ respectively. Let $\tilde\beta$ be the $\mathfrak{L}_{\Gamma}$-isomorphism which extends $\beta$ to $\cl{T}(A,H,a)$ and maps $a$ to $a'$. By conditions
\eqref{pureness} and \eqref{pureness2} we get for every $h\in G$ that $h \in H_G\langle a \rangle$ if and only if
$\tilde\beta(h) \in H'_{G'}\langle a' \rangle$. Hence $\tilde\beta$ is an isomorphism of
$\big(\cl{T}(A,H,a),H_G\langle a \rangle\big)$ and
$\big(\cl{T}(A',H',a'),H'_{G'}\langle a' \rangle\big)$ and $\tilde\beta \in \mathcal{S}$.\\
\underline{Case 2:} $a \in \cl{T}(A,G)$.\\
Let $g_1,\dots,g_n \in G$ be such that $a \in \cl{T}(A,\{g_1,\dots,g_n\}\big)$.
By applying the previous case $n$ times, we get a $\tilde\beta \in \mathcal{S}$ such that
$g_1,\dots,g_n \in \operatorname{dom}(\tilde\beta)$ and $A \subseteq \operatorname{dom}(\tilde\beta)$. Since
$\operatorname{dom}(\tilde\beta)$ is a model of $T_{\Gamma}$, we have $a\in \operatorname{dom}(\tilde\beta)$ with $\tilde\beta\in \mathcal{S}$.\\
\underline{Case 3:} $a \notin \cl{T}(A,G)$.\\
Let $C$ be the cut of $a$ in $\cl{T}(A,H)$ and let $C'$ be the
corresponding cut of $C$ under $\beta$ in $\cl{T}(A',H')$. By
saturation, we can assume that there are $p,q \in M'$ such
that every element in the interval $(p,q)$ realizes the cut $C'$.
Let $d \in M^{|A|}$ be the set $A$ written as a tuple. Let
$f_1,\dots,f_n$ be $\emptyset$-definable functions in the language $\mathfrak{L}_{\Gamma}$. By (A6), we
know that there exists $b \in (p,q)$ such that for $i=1,\dots,n$ and
every tuple $g_1,\dots,g_l$ of elements of $G'$
\begin{equation*}
f_i(g_1,\dots,g_l,d) \neq b.
\end{equation*}
 Thus by saturation, there is an $a' \in (p,q)$ such that $a'
\notin \cl{T}(A',G')$. Since $a'$ realizes the cut $C'$, there is an
$\mathfrak{L}_{\Gamma}$-isomorphism $\tilde\beta$ from $\cl{T}(A,a,H)$ to
$\cl{T}(A',a',H')$ extending $\beta$ and sending $a$ to $a'$. Since $a
\notin \cl{T}(A,G)$ and $a' \notin \cl{T}(A',G')$, we get that
\begin{equation*}
\cl{T}(A,a,H) \cap G = H \textrm{ and } \cl{T}(A',a',H')\cap G' =
H'.
\end{equation*}
Since $\beta(H)=H'$ and $\tilde\beta$ extends $\beta$, we get that
$\tilde\beta$ is an $\mathfrak{L}_{\Gamma}(U)$-isomorphism from $(\cl{T}(A,a,H),H)$ to
$(\cl{T}(A',a',H'),H')$ with $\tilde\beta(A\cup \{a\})=A'\cup\{a'\}$. Thus we
have that $\tilde\beta \in \mathcal{S}$.
\end{proof}

\begin{thm} Assume Condition \ref{schanuel}. Then $\tilde{T}$ is complete.
\end{thm}
\begin{proof}
Let $(M,G)$ and $(M',G')$ be two $|\Gamma|^{+}$-saturated models of $\tilde{T}$, and let $\mathcal{S}$ be as above. It only remains to show that $\mathcal{S}$ is non-empty. But it is easy to see that the identity map on $\cl{T}(\Gamma)$ belongs to $\mathcal{S}$.
 \end{proof}

\end{subsection}

\subsection{Quantifier elimination} In this subsection Theorem \ref{nearMC2} is finally proved (see page \pageref{nearMC2} for the statement).
\begin{proof}[Proof of Theorem \ref{nearMC2}] Let $(M,G)$ and $(M',G')$ be two $|\Gamma|^{+}$-saturated models of $\tilde{T}^{+}$ and let  $\mathcal{S}$ be the back-and-forth system between $(M,G)$ and $(M',G')$ constructed above.
Also take $a=(a_1,\dots,a_n) \in M^n$
and $b=(b_1,\dots,b_n) \in (M')^n$ satisfying the same
quantifier-free $\mathfrak{L}_{\Gamma}(U)^+$-type.  In order to prove quantifier elimination, we just need to find $\tilde\beta \in \mathcal{S}$ sending $a$ to $b$.
 Without loss of generality, we can assume
that $a_1,\dots,a_r$ are maximally
independent over $G$ in respect to the pregeometry $\cl{T}$. Since $a$ and $b$ have the same $\mathfrak{L}_{\Gamma}(U)^+$-type, we get that
$b_1,\dots,b_r$ are independent over $G'$ in respect to the pregeometry $\cl{T}$.  Let $\beta$ be the $\mathfrak{L}_{\Gamma}$-isomorphism between
$\cl{T}(\{a_1,\dots,a_r\},\Gamma)$ and $\cl{T}(\{b_1,\dots,b_r\},\Gamma)$. We will now show that $\beta$ extends to an isomorphism $\tilde\beta$ in the back-and-forth-system $\mathcal{S}$
sending $a$ to $b$. Let $g_1,\dots,g_l \in G$ be such that $a_{r+1},\dots,a_n$ are in $\cl{T}(\{a_1,\dots,a_r,g_1,\dots,g_l\},\Gamma)$.
Let $p(x_1,\dots,x_l)$ be the $\mathfrak{L}_{\Gamma}(U)$-type consisting of the
$\mathfrak{L}_{\Gamma}$-type of $(g_1,\dots,g_l)$ over $\cl{T}(\{a_1,\dots,a_r\},\Gamma)$ and for every $k_1,\dots,k_l\in\Z$, $s\in\N$ and $\gamma\in\Gamma$ one of the formulas
\begin{align}\label{qepureness}
x_1^{k_1}\cdot ... \cdot x_l^{k_l} \cdot \gamma &\in G^{[s]},\\
x_1^{k_1}\cdot ... \cdot x_l^{k_l} \cdot \gamma &\notin G^{[s]},
\label{qepureness2}
\end{align}
depending on whether $g_1^{k_1}\cdot ... \cdot g_l^{k_l} \cdot \gamma \in G^{[s]}$. Let
$p'$ be the type corresponding to $p$ under $\beta$. We want to find
$h_1,\dots,h_l \in G'$ satisfying $p'$. By compactness and saturation of
$(M',G')$, it is enough to show that every finite subset of
$p'$ can be realized. So let $\psi(x,b_1,\dots,b_r)$ be an $\mathfrak{L}_{\Gamma}$-formula  in $p'$ and $\chi_1(x,b_1,\dots,b_r),\dots,\chi_t(x,b_1,\dots,b_r)$ be finitely many formulas in $p'$ of the form \eqref{qepureness} or \eqref{qepureness2}. Put $\chi=\bigwedge_{i=1}^{t}\chi_i$. By Lemma \ref{unionofcosets}, the set
$$Y:=\big\{ (h_1,\dots,h_l) \in G'^{l} \ : \ (M',G')\models\chi(h_1,\dots,h_l,b_1,\dots,b_r))\big\}$$
is a finite union of cosets of $(G'^{[s]})^l$ in $(G')^l$ for some $s \in \mathbb{N}$. So the formula $\chi_i(x,b_1,\dots,b_r)$ is equivalent to an atomic $\mathfrak{L}_{\Gamma}(U)^{+}$-formula. Hence the formula $\psi\wedge \chi$ is also of this form. Thus
\begin{equation*}
\exists y_1 \cdots \exists y_l \bigwedge_{i=1}^{l} U(y_i) \wedge \psi(y_1,\dots,y_l,b_1,\dots,b_r) \wedge \chi(y_1,\dots,y_l,b_1,\dots,b_r)
\end{equation*}
is equivalent to a quantifier-free $\mathfrak{L}_{\Gamma}(U)^{+}$-formula. Since $(a_1,\dots,a_r)$ and $(b_1,\dots,b_r)$ have the same quantifier-free $\mathfrak{L}_{\Gamma}(U)^{+}$-type, the formula
\begin{equation*}
\exists y_1 \cdots \exists y_l \bigwedge_{i=1}^{l} U(y_i) \wedge \psi(y_1,\dots,y_l,b_1,\dots,b_r) \wedge \chi(y_1,\dots,y_l,b_1,\dots,b_r)
\end{equation*}
holds in $(M',G')$. So $p'$ is finitely satisfiable. Now let $h_1,\dots,h_l \in G'$ realize $p'$. Then $\beta$ extends to an $\mathfrak{L}_{\Gamma}$-isomorphism
$$\tilde\beta:\\ \cl{T}\big(\{a_1,\dots,a_r,g_1,\dots,g_l\},\Gamma\big)\to\cl{T}\big(\{b_1,\dots,b_r,h_1,\dots,h_l\},\Gamma\big).$$
By the construction of $g_1,\dots,g_l$ and $h_1,\dots,h_l$, we have that
\begin{equation*}
g_1^{k_1} \cdot ... \cdot g_l^{k_l} \gamma \in G^{[s]} \textrm{ if and only if } h_1^{k_1} \cdot ... \cdot h_l^{k_l} \gamma \in G'^{[s]}
\end{equation*}
for all $k_1,\dots,k_l \in \mathbb{Z}$, $s \in \mathbb{N}$ and $\gamma \in \Gamma$. Hence $\tilde\beta$ is an $\mathfrak{L}_{\Gamma}(U)$-isomorphism of  $$\big(\cl{T}(\{a_1,\dots,a_r,g_1,\dots,g_l\},\Gamma),\Gamma_G \< g_1,\dots,g_l\> \big)\text{ and }$$ $$\big(\cl{T}(\{b_1,\dots,b_r,h_1,\dots,h_l\},\Gamma),\Gamma_{G'}\<h_1,\dots,h_l\>\big).$$
Hence $\tilde\beta \in \mathcal{S}$.
\end{proof}

\subsection{Induced structure and open core} In this subsection it will be shown that every open definable set in
$(\overline{\R},x^{\tau},\tau,\Gamma)$ is already definable in the reduct $(\overline{\R},x^{\tau},\tau)$. This establishes the second part of Theorem A. We use the following instance of \cite{ayhanme2}, Theorem 5.2.

\begin{thm}\label{ocayhan} Suppose that for every model $(M,G)\models \tilde{T}$,
\begin{itemize}
\item for every finite $B\subseteq M$ such that $B\setminus G$ is $\cl{T}$-independent over $G$ and
\item for every
 set $X \subseteq G^n$ definable in $(M,G)$ with parameters from $B$,
\end{itemize}
the topological closure $\overline{X}$ of $X$ is definable in $M$ over $B$. Then every open set definable in $(\overline{\R},x^{\tau},\tau,\Gamma)$ is already definable in $(\overline{\R},x^{\tau},\tau)$.
\end{thm}

\noindent In the remainder it will be shown that the assumption of Theorem \ref{ocayhan} holds. Therefor let $(M,G)$ be a model of $\tilde{T}$ and let $B$ be a finite subset of $M$ such that $B\setminus G$ is $\cl{T}$-independent over $G$.

\begin{lem}\label{induced}
Let $X \subseteq G^n$ be definable in $(M,G)$ with parameters from $B$. Then $X$ is a finite union of sets of the form
\begin{equation}\label{inducedform}
E \cap \bigcup_{i=1}^{l} \gamma_i \cdot (G^{[s]})^n.
\end{equation}
where $E\subseteq M^{n}$ is $\mathfrak{L}_{\Gamma}$-$B$-definable, $\gamma_1,\dots,\gamma_l\in\Gamma^n$ and $s\in\N$.
\end{lem}
\begin{proof} We may assume that $(M,G)$ is a $|\Gamma|^{+}$-saturated model of $\tilde{T}$. By our assumption, $B$ is a union of a finite set $S \subseteq G$ and a set $A\subseteq M$ which is $\cl{T}$-independent over $G$.
Let $\mathcal{S}$ be the back-and-forth
system of partial $\mathfrak{L}_{\Gamma}(U)$-isomorphisms between $(M,G)$ and itself constructed above. Take $g,g'\in G^n$ such that
for every $E\subseteq M^{n}$ $\mathfrak{L}_{\Gamma}$-definable over $B$, $\gamma_1,\dots,\gamma_l\in\Gamma^n$ and $s\in\N$ we have that
\begin{equation}\label{inducedeq}
g\in E \cap \bigcup_{i=1}^{l} \gamma_i (G^{[s]})^n\Leftrightarrow g'\in E \cap \bigcup_{i=1}^{l} \gamma_i (sG^{[s]})^n.
\end{equation}
By Lemma \ref{unionofcosets} and (A3), the collection of finite union of sets of the form \eqref{inducedform} is closed under boolean operations. Hence it suffices to show that there is $\beta\in \mathcal{S}$ fixing $B$ such that $\beta$ maps $g$ to
$g'$. Since $g$ satisfies all $\mathfrak{L}_{\Gamma}$-formulas over $B$ which are satisfied by $g$, there is an $\mathfrak{L}_{\Gamma}$-isomorphism from $\cl{T}(g,B,\Gamma)$ to $\cl{T}(g',B,\Gamma)$ fixing $B\cup \Gamma$ and mapping $g$ to $g'$. We now show that $\beta \in \mathcal{S}$. Since $B=S\cup A$, it is only left to prove that $\beta(\Gamma\langle g,S \rangle) = \Gamma\langle g',S\rangle$. Since $\beta$ maps $g$ to $g'$ and fixes $S$, it is enough to show for all $h \in \Gamma_G\langle S \rangle^n$, $k \in \mathbb{Z}^n$ and $s \in \mathbb{N}$ that
\begin{equation*}
g \in h \cdot D_{k,s} \textrm{ if and only if } g' \in h \cdot  D_{k,s}.
\end{equation*}
By Lemma \ref{unionofcosets} and (A3), there is are $\gamma_1,\dots,\gamma_{l_1},\delta_1,\dots,\delta_{l_2} \in \Gamma^n$ such that \newline $h \cdot D_{k,s}= \bigcup_{i=1}^{l_1} \gamma_i (G^{[s]})^n$ and $G^n\setminus (h \cdot \gamma D_{k,s}) = \bigcup_{i=1}^{l_2} \delta_i (G^{[s]})^n$. By \eqref{inducedeq}, we have $g \in h \cdot D_{k,s}$ if and only if $g' \in h \cdot D_{k,s}$. Hence $\beta(\Gamma\langle g,S \rangle) = \Gamma\langle g',S\rangle$ and $\beta \in \mathcal{S}$.
\end{proof}

\begin{prop}\label{induced-corollary}Let $X \subseteq G^n$ be definable in $(M,G)$ with parameters from $B$. Then the topological closure $\overline{X}$ of $X$ is definable in $M$ over $B$.
\end{prop}
\begin{proof} We prove that there is an $\mathfrak{L}_{\Gamma}$-$B$-definable set $E\subseteq M^{n}$  such that $X$ is a dense subset of $E$. We do this by induction on $n$. The case $n=0$ is trivial. So let $n>0$. By Lemma \ref{induced} we may assume that there exists an $\mathfrak{L}_{\Gamma}$-$B$-definable set $E_0$ and an $\mathfrak{L}_{\Gamma}(U)$-$\emptyset$-definable set $D_0$ which is dense in $G^n$ such that $X=E_0 \cap D_0$. Without loss of generality, we can assume that $E_0$ is a cell.
First consider the case that $E_0$ is open. Then $X$ is dense in $E_0$. Now consider the case that there is a projection $\pi : M^{n} \to M^m$ such that $m<n$ and $\pi$ is homeomorphism of $E_0$ onto its image and $\pi(E_0)$ is open.
Consider the set
\begin{align*}
X':= \{ (g_1,\dots,g_{m}) \in G^m \cap \pi(E_0) \ : \ \pi^{-1}(g_1,\dots,g_{m}) \in D_0\}.
\end{align*}
By the induction hypothesis, there is an $\mathfrak{L}_{\Gamma}$-$B$-definable subset $E_1$ of $\pi(E_0)$  such that $X'$ is dense in $E_1$. By continuity of $\pi^{-1}$, the image of $X'$ under $\pi^{-1}$ is dense in the image of $E_1$ under $\pi^{-1}$. Set $E:=\pi^{-1}(E_1)$. Since $X=\pi^{-1}(X')$, we have that $X$ is dense $E$.
\end{proof}

\subsection{Proof of Theorem B} As mentioned above, the proof of Theorem B, ie. the case $\tilde{T}=T_{\Gamma}(\Delta)$, is almost exactly the same as the proof of Theorem A. One only needs to replace `$H$ is a pure subgroup of $G$' by `$H$ is a $\qt$-linear subspace of $G$' in the statement of Lemma \ref{mainlemma} and the definition of the back-and-forth system $\mathcal{S}$, and adjust the proof of Lemma \ref{bnf} and Theorem \ref{nearMC2} accordingly.

\end{section}

 

\end{document}